\documentclass{amsart}
\usepackage{amsfonts,amsthm,amsmath}
\usepackage{amssymb}
\usepackage[usenames,dvipsnames]{color}
\usepackage[utf8]{inputenc}
\usepackage{xypic}
\usepackage[a4paper]{geometry}
\usepackage{hyperref}
\usepackage{pgfplots}
\usepackage{listings}

\DeclareMathOperator{\res}{res} 
\DeclareMathOperator{\diag}{diag} 
\theoremstyle{plain}
	\newtheorem{thm}{Theorem}
	
	\newtheorem{lemma}[thm]{Lemma}

\theoremstyle{remark}
	\newtheorem{remark}[thm]{Remark}

\theoremstyle{definition}
		
\title{On Chern classes of tensor products of vector bundles}
\author{Zsolt Szilágyi}
\address{Babe\c s-Bolyai University, Faculty of Mathematics and Computer Science, Kog\u alniceanu Street 1, 
400084, Cluj-Napoca, Romania
}
\email{szilagyi.zsolt@math.ubbcluj.ro}
\allowdisplaybreaks 
\begin{document}
\begin{abstract}
We present two  formulas for  Chern classes of the tensor product of two vector bundles. 
In the first formula we consider a matrix containing Chern classes of the first bundle and we take a polynomial of this matrix with Chern classes of the second bundle as coefficients. The determinant of this expression equals the Chern polynomial of the tensor product.
In the second formula we express the total Chern class of the tensor product as resultant of two explicit polynomials with coefficient involving Chern classes of each vector bundles.  This approach leads to determinantal formulas for the total Chern class of the second symmetric and alternating product of a vector bundle. 
\end{abstract}

\maketitle

\section{Introduction}

\subsection{Chern classes of vector bundles}
\label{s:Chern classes}
One associates  a series of cohomological (characteristic) classes $c_{i}(\mathcal{E})\in H^{2i}(M)$ called the $i^{th}$ \emph{Chern class} of $\mathcal{E}$, for all $i=1,\dots ,r$, 
with a complex vector bundle $\mathcal{E}$ of rank $r$ over a manifold $M$ 
 (cf.~\cite[Ch.~IV]{bott-tu-book} or \cite[Ch.~I, \S4]{hirzebruch1978topological}). 
One can arrange them into a polynomial  $c(\mathcal{E};t) = 1 + c_{1}(\mathcal{E})t + \dots + c_{r}(\mathcal{E})t^{r}$, called the \emph{Chern polynomial}.   Its value $c(\mathcal{E};1) =  1 + c_{1}(\mathcal{E}) + \dots + c_{r}(\mathcal{E})$ at $t=1$ is the \emph{total Chern class}.  
In the sequel we only consider complex vector bundles over the fixed manifold $M$.

We recall some properties of the Chern classes. 
If $\mathcal{F}$ is another vector bundle  then by the \emph{Whitney product formula} $c(\mathcal{E} \oplus \mathcal{F};t) =  c(\mathcal{E};t)\cdot c(\mathcal{F};t)$ (cf. \cite[(20.10.3)]{bott-tu-book}).
Computing with Chern classes one can pretend by the \emph{Splitting Principle} (cf. \cite[Ch.~IV, \S21]{bott-tu-book}) that the bundle $\mathcal{E}$ of rank $r$ splits into direct sum of $r$ complex line bundles   and the first Chern classes $\alpha_{1},\dots ,
\alpha_{r}$ of these hypothetical line bundles are the so-called \emph{Chern roots} of $\mathcal{E}$. Hence, by the Whitney product formula $c(\mathcal{E};t) = \prod_{i=1}^{r}(1 + \alpha_{i}t)$, thus $c_{k}(\mathcal{E}) = e_{k}(\alpha) = e_{k}(\alpha_{1},\dots ,\alpha_{r})=\sum_{1\leq i_{1}< \dots < i_{k}\leq r} \alpha_{i_{1}}\cdots \alpha_{i_{k}}$, $k=1,\dots ,r$, i.e. the Chern classes  are \emph{elementary symmetric polynomials} of the Chern roots. 

The Chern polynomial does not behave so well for tensor product like for the direct sum. 
Nevertheless, for complex line bundles $\mathcal{L}$ and $\mathcal{L}'$ we have 
$c_{1}(\mathcal{L} \otimes \mathcal{L}') = c_{1}(\mathcal{L}) + c_{1}(\mathcal{L}')$ (cf. \cite[(20.1)]{bott-tu-book}). 
Hence, if $\alpha_{1},\dots ,\alpha_{r}$ and $\beta_{1},\dots ,\beta_{q}$ are Chern roots of $\mathcal{E}$ and $\mathcal{F}$, respectively, then $\alpha_{i} + \beta_{j}$, $i=1,\dots,r$, $j=1,\dots ,q$ are the Chern roots of the tensor product $\mathcal{E}\otimes \mathcal{F}$ and the Chern polynomial of the tensor product equals 
\begin{equation}\label{eq-Chern-tensor-def}
c(\mathcal{E}\otimes \mathcal{F};t) = \prod_{i=1}^{r}\prod_{j=1}^{q}(1 + \alpha_{i}t + \beta_{j}t).
\end{equation}
The goal is to express \eqref{eq-Chern-tensor-def} in terms of 
Chern classes of $\mathcal{E}$ and $\mathcal{F}$, or equivalently in terms of elementary symmetric polynomials of $\alpha$'s and $\beta$'s, respectively.

\subsection{Existing computing methods}
\label{ss:existing methods}
There are several approaches to compute the Chern classes of the tensor product. We mention   the four approaches compared in \cite{iena2016different}. 

The first method computes the Chern classes of the tensor product by eliminating Chern roots $\alpha_{1},\dots ,\alpha_{r},\beta_{1},\dots ,\beta_{q}$ from $c(\mathcal{E}\otimes \mathcal{F}) = \prod_{i=1}^{r} \prod_{j=1}^{q} (1 + \alpha_{i} + \beta_{j})$ using relations $c_{i}(\mathcal{E}) = e_{i}(\alpha_{1},\dots ,\alpha_{r})$ and $c_{j}(\mathcal{F}) = e_{j}(\beta_{1},\dots ,\beta_{q})$ for $i=1,\dots ,r$ and $j=1,\dots ,q$. The second approach uses the multiplicativity of the Chern character (cf.~\cite[Ch.~III, \S10.1]{hirzebruch1978topological}) and Newton's identities (cf.~\cite[($2.11'$)]{macdonald1995symmetric}). The third uses Lascoux's formula \cite{lascoux1978classes} which expresses the Chern classes of the tensor product as linear combination of products of Schur polynomials of $\mathcal{E}$ and $\mathcal{F}$. The last approach is Manivel's formula \cite{manivel2016chern}, which is a variation of Lascoux's formula computing the coefficients differently. These methods have been implemented in the library \textsc{chern.lib} \cite{iena2016chernlib} for the computer algebra system \textsc{Singular} \cite{Singular4-1-2}.

\section{First approach}

\begin{lemma}\label{lm-0}
Let $s$, $u_{1},\dots ,u_{r}$, $v_{1},\dots,v_{q}$ be formal variables and 
 $e_{k}(u) = \sum_{1\leq i_{1}< \dots < i_{k}\leq r} u_{i_{1}}\cdots u_{i_{k}}$ is the $k^{th}$ elementary symmetric polynomial. 
We associate with  $e(v) = (e_{1}(v),\dots ,e_{q}(v))$ the following matrix
\begin{equation}\label{eq-Lambda}
\Lambda(e(v)) 
= 
 \begin{pmatrix} 
 e_{1}(v) & -1 &  & 
 \\
 \vdots & &\ddots & 
 \\
 e_{q-1}(v) & &  & -1
 \\
e_{q}(v) &  &  & 
  \end{pmatrix}
\end{equation} 
(it has non-zero entries only in the first column and above the diagonal).
Then we have 
$$
\prod_{i=1}^{r} \prod_{j=1}^{q} (s+u_{i}+v_{j})
=
\det\left( \sum_{k=0}^{r} e_{k}(u)[sI+\Lambda(e(v))]^{r-k} \right),
$$
where $I=I_{q}$ is the $q$-by-$q$ identity matrix.

\end{lemma}

\begin{proof}
First, we diagonalize the matrix $\Lambda(e(v))$. Therefore, we consider the $q$-by-$q$ matrix 
$$
E=E(v_{1},\dots ,v_{q}) = \big[ e_{i-1}(v_{1},\ldots ,\widehat{v}_{j},\ldots ,v_{q}) \big]_{i,j=1}^{q},
$$
where $e_{0}(v_{1},\dots ,\widehat{v}_{j},\dots ,v_{q})=1$ and $\widehat{v}_{j}$ means that $v_{j}$ is omitted. We show that $E$ is non-singular by computing its determinant as follows.
We subtract the first column from the other columns, then we raise a $(v_{1}-v_{j})$-factor from columns $j=2,\dots ,q$, respectively.
Expanding the resulting determinant by the first row we get the recurrent relation $\det(E(v_{1},\dots ,v_{q})) = \prod_{j=2}^{q}(v_{1}-v_{j}) \det(E(v_{2}, \dots , v_{q}))$, hence  $\det(E) =  \prod_{1\leq i<j \leq q}(v_{i} - v_{j})\neq 0$. 
Moreover, 
$$
\Lambda(e(v)) E = E\,\textnormal{diag}(v_{1},\dots ,v_{q})
$$ 
by relations $e_{i} (v_{1},\dots , v_{q})  = e_{i}(v_{1},\ldots ,\widehat{v}_{j},\ldots ,v_{q}) + v_{j}e_{i-1}(v_{1},\ldots ,\widehat{v}_{j},\ldots ,v_{q})$, hence $\Lambda(e(v))=E \diag(v_{1},\dots ,v_{q}) E^{-1}$.
Furthermore, $sI+\Lambda(e(v)) =sI+E \diag(v_{1},\dots ,v_{q})E^{-1} = E \diag(s+v_{1},\ldots ,s+v_{q}) E^{-1}$ is also diagonalizable with eigenvalues $s+v_{1},\ldots ,s+v_{q}$.
Finally,
\begin{gather*}
\prod_{j=1}^{q}\prod_{i=1}^{r}(s+u_{i}+v_{j})
=
\prod_{j=1}^{q}\sum_{k=0}^{r} e_{k}(u) (s+v_{j})^{r-k}
=
\\
=
\det \bigg( E \diag \bigg( \sum_{k=0}^{r} e_{k}(u) (s + v_{1})^{r-k},\dots ,\sum_{k=0}^{r} e_{k}(u)(s+v_{q})^{r-k} \bigg) E^{-1} \bigg)
=
\\
=
\det \bigg(  \sum_{k=0}^{r} e_{k}(u) E\diag \big(  s + v_{1},\dots ,s+v_{q} \big)^{r-k} E^{-1} \bigg)
=
\det \bigg(  \sum_{k=0}^{r} e_{k}(u) [sI+\Lambda(e(v))]^{r-k} \bigg).
\end{gather*}
\end{proof}

\begin{thm}\label{thm-Chern-formula-1}
Let $\mathcal{E}$ and $\mathcal{F}$ be two complex vector bundles of rank $r$ and $q$, respectively. 
The Chern polynomial of the tensor product $\mathcal{E}\otimes \mathcal{F}$ equals
\begin{equation}\label{eq-tensor-Chern}
c(\mathcal{E}\otimes \mathcal{F};t)
=
\det \Bigg( \sum_{k=0}^{r} c_{k}(\mathcal{E})t^{k}[I+\Lambda(c(\mathcal{F});t)]^{r-k} \Bigg),
\end{equation}
where 
$c_{0}(\mathcal{E}) = 1$ and $\Lambda(c(\mathcal{F});t)$ is the matrix \eqref{eq-Lambda} with $c_{1}(\mathcal{F})t,\dots ,c_{q}(\mathcal{F})t^{q}$ in the first column.
\end{thm}
\begin{proof}
Let $\alpha_{1},\dots ,\alpha_{r}$ and $\beta_{1},\dots ,\beta_{q}$ be the Chern roots of  $\mathcal{E}$ and $\mathcal{F}$, respectively. Then it is enough to show that
$
\prod_{i=1}^{r} \prod_{j=1}^{q}(1 + \alpha_{i}t + \beta_{j}t) 
= 
\det \left( \sum_{k=0}^{r} e_{k}(\alpha)t^{k}[I+\Lambda(e(\beta t))]^{r-k} \right),
$
where $\Lambda(e(\beta t))$ equals the matrix $\Lambda(c(\mathcal{F});t)$ only replacing Chern classes $c_{j}(\mathcal{F})$  by elementary symmetric polynomials $e_{j}(\beta) = e_{j}(\beta_{1},\dots ,\beta_{q})$ of Chern roots for all $j=1,\dots ,q$. Finally, substituting $s=1$, $u_{1}=\alpha_{1}t,\dots ,u_{r}=\alpha_{r}t$, $v_{1} = \beta_{1} t,\dots ,v_{q} = \beta_{q} t$ in Lemma \ref{lm-0} yields the desired relation.
\end{proof}

\section{Second approach: resultant and Chern classes of the tensor product}

Our second approach uses resultant of two polynomials. This will lead us to a determinantal formula for Chern classes of second alternating and symmetric products.

\subsection{Resultant}

Let $A(t) = a_{r} + a_{r-1} t + \dots + a_{0}t^{r} = a_{0} \prod_{i=1}^{r}(t-\alpha_{i})$  and $B(t) = b_{q} + b_{q-1} t + \dots + b_{0} t^{q} = b_{0} \prod_{j=1}^{q}(t-\beta_{j})$ two polynomials in the variable $t$ with roots $\alpha_{1},\dots ,\alpha_{r}$ and $\beta_{1},\dots ,\beta_{q}$, respectively. The \emph{resultant} of polynomials $A$ and $B$ with respect to variable $t$ is given by 
\begin{align*}
\res(A(t),B(t),t)  &{}= a_{0}^{q} b_{0}^{r} \prod_{i=1}^{r} \prod_{j=1}^{q}(\alpha_{i}-\beta_{j}) 
=
a_{0}^{q} \prod_{i=1}^{r} B(\alpha_{i}) 
= 
b_{0}^{r} \prod_{j=1}^{q} A(\beta_{j}) 
\\
&{}=
\begin{vmatrix}
a_{0} &  &  &   &  b_{0} &  &  & 
\\
a_{1} & a_{0} &    &   & b_{1} & b_{0} &  & 
\\
\vdots &  \vdots & \ddots &  & \vdots & \vdots  & \ddots & 
\\
a_{r} & a_{r-1} & &  a_{0}   & b_{q} & b_{q-1} &  & b_{0}
\\
 & a_{r} &  \ddots & \vdots  &  & b_{q} & \ddots & \vdots 
\\
 &  & \ddots & a_{r-1} &   &  & \ddots & b_{q-1} 
\\
 &  &  &    a_{r} &    &  &   & b_{q} 
\end{vmatrix}
\end{align*}
where the first $q$ columns contain the coefficients of $A$, while the last $r$ columns contain the coefficients of $B$ and empty spaces contain zeroes (cf.~\cite[Ch.~III]{cox2005using}).

\subsection{Two polynomials}
Instead of the Chern polynomial $c(\mathcal{F};t) = 1 + c_{1}(\mathcal{F})t+ \dots + c_{q}(\mathcal{F})t^{q}$  
of the rank $q$ vector bundle $\mathcal{F}$
we consider the polynomial with coefficients in  reverse order
\begin{equation}\label{eq-poly-C}
C(\mathcal{F};t) 
= \sum_{k=0}^{q} c_{k}(\mathcal{F}) t^{q-k} 
= c_{q}(\mathcal{F}) + c_{q-1}(\mathcal{F}) t + \dots + c_{1}(\mathcal{F})t^{q-1} + t^{q}. 
\end{equation}
They are related  by $C(\mathcal{F};t) = t^{q} c(\mathcal{F}; 1/t)$ and moreover, we can recover the total Chern class by substituting $t=1$, i.e. $c(\mathcal{F}) = C(\mathcal{F};1)$. Furthermore, if $\beta_{1},\dots,\beta_{q}$ are Chern roots of $\mathcal{F}$ then 
$
C(\mathcal{F};t) = \prod_{j=1}^{q}(t+\beta_{j})
$, 
i.e. the opposite of Chern roots of $\mathcal{F}$ are roots of the polynomial $C(\mathcal{F};t)$.

We also consider another polynomial (depending on an additional parameter $s$) associated with the Chern classes of the vector bundle $\mathcal{E}$ of rank $r$.
Denote  
\begin{equation}\label{eq-poly-D}
D(\mathcal{E};s,t) 
= 
(-1)^{r}d_{r}(\mathcal{E};s) + (-1)^{r-1}d_{r-1}(\mathcal{E};s) t + \dots + d_{0}(\mathcal{E};s)t^{r}
=
\sum_{k=0}^{r} (-1)^{k} d_{k}(\mathcal{E};s) t^{r-k}
\end{equation}
with coefficients 
\begin{equation}\label{eq-coeffs-ds}
d_{k}(\mathcal{E};s) = \binom{r}{k} s^{k} + \binom{r-1}{k-1} c_{1}(\mathcal{E}) s^{k-1} + \dots 
+ c_{k}(\mathcal{E})
=
\sum_{i=0}^{k} \binom{r-i}{k-i}c_{i}(\mathcal{E}) s^{k-i}.
\end{equation}
Remark that $d_{k}(\mathcal{E};0) = c_{k}(\mathcal{E})$, thus $D$ is a generalization of $C$, more precisely, we have $C(\mathcal{E};t) = (-1)^{r} D(\mathcal{E};0,-t)$ and moreover, $(-1)^{r}D(\mathcal{E};s,0)=d_{r}(\mathcal{E};s) = C(\mathcal{E};s)$. 
\begin{lemma}\label{Lm-1}
If $\alpha_{1},\dots,\alpha_{r}$ are the Chern roots of $\mathcal{E}$ then 
$
D(\mathcal{E};s,t) = \prod_{i=1}^{r}(t - s - \alpha_{i}).
$
\end{lemma}
\begin{proof}
Indeed, $\prod_{i=1}^{r}(t - s - \alpha_{i})= \sum_{k=0}^{r} (-1)^{k} e_{k}(s+\alpha_{1},\dots ,s+\alpha_{r}) t^{r-k}$ and
\begin{gather*}
\label{eq-computation-1}
e_{k}(s+\alpha_{1},\dots , s+\alpha_{r}) 
= 
\sum_{1\leq i_{1}<\dots <i_{k}\leq r} (s + \alpha_{i_{1}}) \cdots (s + \alpha_{i_{k}})=
\\
\nonumber
=
\sum_{1\leq i_{1}<\dots <i_{k}\leq r} \left[ s^{k} + e_{1}(\alpha_{i_{1}},\dots ,\alpha_{i_{k}})s^{k-1} +\dots + e_{k}(\alpha_{i_{1}},\dots ,\alpha_{i_{k}})  \right]
=
\\
\nonumber
=\binom{r}{k} s^{k} + \binom{r-1}{k-1}  e_{1}(\alpha_{1},\dots ,\alpha_{r}) s^{k-1}
+ \dots + \binom{r-k}{0} e_{k} (\alpha_{1},\dots,\alpha_{r})
=
\\
\nonumber
=
\sum_{i=0}^{r} \binom{r-i}{k-i}  e_{i}(\alpha_{1},\dots ,\alpha_{r}) s^{k-i}
=
\sum_{i=0}^{k} \binom{r-i}{k-i}  c_{i}(\mathcal{E}) s^{k-i}
= d_{k}(\mathcal{E};s).
\qedhere
\end{gather*} 
\end{proof}

\subsection{Chern classes of the tensor product via resultant}

In the next theorem we express $C(\mathcal{E}\otimes \mathcal{F};s)$ as resultant of polynomials $D(\mathcal{E}; s,t)$ and $C(\mathcal{F};t)$.
 We can also get a formula for the total Chern class of the tensor product by substituting $s=1$ into the formula. 
\begin{thm}\label{thm-v1}
If $\mathcal{E}$ and $\mathcal{F}$ are two complex vector bundles of rank $r$ and $q$, respectively, then 
\begin{equation}\label{eq-C}
C(\mathcal{E} \otimes \mathcal{F}; s) = \res(D(\mathcal{E};s,t),C(\mathcal{F};t),t),
\end{equation}
where polynomials $C$ and $D$ are defined by \eqref{eq-poly-C} and \eqref{eq-poly-D}, respectively. Substituting $s=1$ yields
$$
c(\mathcal{E}\otimes \mathcal{F}) = \res(D(\mathcal{E};1,t),C(\mathcal{F};t),t).
$$
Moreover, the top Chern class of the tensor product equals
$$
c_{r+q}(\mathcal{E}\otimes \mathcal{F}) 
= 
\res((-1)^{r} C(\mathcal{F};-t), C(\mathcal{E};t),t) 
= 
\res(c(\mathcal{E};t),c(\mathcal{F};-t), t).
$$
\end{thm}
\begin{proof}
Denote $\alpha_{1},\dots ,\alpha_{r}$ and $\beta_{1},\dots ,\beta_{q}$ the Chern roots of $\mathcal{E}$  and $\mathcal{F}$, respectively. Then 
$$
C(\mathcal{E}\otimes \mathcal{F};s) 
=
\prod_{i=1}^{r}\prod_{j=1}^{q} (s + \alpha_{i} + \beta_{j}) 
=
\prod_{i=1}^{r}\prod_{j=1}^{q} (s + \alpha_{i}  -(-\beta_{j}))
$$
and $C(\mathcal{E}\otimes \mathcal{F};s)$ is the resultant of polynomials $\prod_{i=1}^{r}(t-s-\alpha_{i}) = D(\mathcal{E};s,t)$  (cf.~Lemma \ref{Lm-1}) and $\prod_{j=1}^{q}(t+\beta_{j}) = C(\mathcal{F};t)$ with respect to the variable $t$, i.e. $C(\mathcal{E} \otimes \mathcal{F}; s) = \res(D(\mathcal{E};s,t),C(\mathcal{F};t),t)$.

To obtain the top Chern class first 
we reverse the roles of $\mathcal{E}$ and $\mathcal{F}$ and we substitute $s=0$ into \eqref{eq-C}. Finally, we reverse the orders of rows and columns in the defining determinant of the resultant. 
\end{proof}

\subsection{Chern classes  $c(\wedge^{2}\mathcal{E})$ and $c(S^{2}\mathcal{E})$}

We give a different version of Theorem \ref{thm-v1}, which leads to 
determinantal 
formulas for total Chern classes of the second alternating and symmetric product.
%
\begin{thm}\label{thm-v2}
If $\mathcal{E}$ and $\mathcal{F}$ are two complex vector bundles of rank $r$ and $q$, respectively, then 
$$
C(\mathcal{E} \otimes \mathcal{F};s) = \res ( D (\mathcal{E};s/2,t),(-1)^{q} D(\mathcal{F};s/2,-t), t),
$$
where the polynomial $D$ is defined in \eqref{eq-poly-D}.
By substituting $s=1$  we get
\begin{equation}
c(\mathcal{E}\otimes \mathcal{F}) = \res( D(\mathcal{E};1/2,t),(-1)^{q}D(\mathcal{F};1/2,-t),t).
\end{equation}
\end{thm}
\begin{proof}
If $\alpha_{1},\dots ,\alpha_{r}$ and $\beta_{1},\dots ,\beta_{q}$ are the Chern roots of $\mathcal{E}$ and $\mathcal{F}$, respectively, then 
$$
C(\mathcal{E} \otimes \mathcal{F}; s) 
= 
\prod_{i=1}^{r}\prod_{j=1}^{q}(s+\alpha_{i}+\beta_{j}) 
=
\prod_{i=1}^{r}\prod_{j=1}^{q}\left( \frac{s}{2} + \alpha_{i} -\left(- \frac{s}{2}-\beta_{j}\right)\right), 
$$
hence $C(\mathcal{E}\otimes \mathcal{F};s)$ is the resultant of polynomials $D(\mathcal{E};s/2,t) = \prod_{i=1}^{r}\left(t - \frac{s}{2} - \alpha_{i} \right)$ and 
\[
(-1)^{q}D(\mathcal{F};s/2,-t) 
= 
(-1)^{q}\prod_{i=1}^{q} \left(-t - \frac{s}{2} - \beta_{j} \right) 
= 
\prod_{i=1}^{q} \left(t + \frac{s}{2} + \beta_{j} \right). 
\qedhere
\]
\end{proof}

\subsubsection{Chern classes of the second alternating product $\wedge^{2}\mathcal{E} $ and second symmetric product $S^{2}\mathcal{E}$}
\label{ss:wedge and symmetric}

Let $\alpha_{1},\dots ,\alpha_{r}$ be the Chern roots of the vector bundle $\mathcal{E}$. Then 
$c(\wedge^{2}\mathcal{E}) = \prod_{1\leq i<j \leq r}(1+\alpha_{i}+\alpha_{j})$
and 
$c(S^{2}\mathcal{E}) = \prod_{1\leq i\leq j \leq r}(1+\alpha_{i}+\alpha_{j}) = c(\mathcal{E};2) c(\wedge^{2}\mathcal{E})$,  
hence
$
C(\wedge^{2} \mathcal{E}; s) = \prod_{1\leq i < j\leq r} (s + \alpha_{i} + \alpha_{j})
\textnormal{ and } 
C(S^{2} \mathcal{E}; s) = \prod_{1\leq i \leq j\leq r} (s + \alpha_{i} + \alpha_{j}) = 2^{r}C(\mathcal{E};s/2) C(\wedge^{2}\mathcal{E};s).
$

\begin{thm}
Let $\bar{d}_{k} = d_{k}(\mathcal{E};s/2) = \sum_{i=0}^{k} \binom{r-i}{k-i}c_{i}(\mathcal{E}) (s/2)^{k-i}$ for $k=0,1,\dots ,r$ and $\bar{d}_{k}=0$ otherwise. With these notations we have 
\begin{equation}\label{eq-wedge2}
C(\wedge^{2}\mathcal{E};s) 
=
\det\left(\left[d_{2i-j}(\mathcal{E};s/2) \right]_{i,j=1}^{r-1}\right)
= \begin{vmatrix}
\bar{d}_{1} & 1 &  &   &   &\\
\bar{d}_{3} & \bar{d}_{2} & \bar{d}_{1} & 1 &  
\\
\bar{d}_{5} & \bar{d}_{4} & \bar{d}_{3} & \bar{d}_{2} & \bar{d}_{1} & \\
\vdots & \ddots & \ddots & \ddots & \ddots & 
\ddots \\
&  & \bar{d}_{r} & \bar{d}_{r-1} & \bar{d}_{r-2} & \bar{d}_{r-3} 
\\
&&&  & \bar{d}_{r} & \bar{d}_{r-1}
\end{vmatrix}.
\end{equation}
By substituting $s=0$ we get $c(\wedge^{2}\mathcal{E}) = \det\left(\left[d_{2i-j}(\mathcal{E};1/2) \right]_{i,j=1}^{r-1}\right)$.
\end{thm}

\begin{proof}
By Theorem \ref{thm-v2} we have $
C(\mathcal{E}\otimes \mathcal{E}; s) 
= \res(D(\mathcal{E};s/2,t),(-1)^{r}D(\mathcal{E};s/2,-t),t)
$, which equals 
\begin{equation}\label{eq-det1}
\begin{vmatrix}
1 & & 
&  1 & &
\\
-\bar{d}_{1} & \ddots & 
& 
\bar{d}_{1} & \ddots & 
\\
\bar{d}_{2} & \ddots & 1  
& 
\bar{d}_{2} & \ddots & 1
\\
\vdots  &  & -\bar{d}_{1} 
& \vdots &   &  \bar{d}_{1}   
\\
(-1)^{r}\bar{d}_{r}  & & \vdots 
& 
\bar{d}_{r} &   & \vdots
\\
 &  \ddots &  (-1)^{r-1}\bar{d}_{r-1}
&& \ddots & \bar{d}_{r-1} 
\\
& & (-1)^{r}\bar{d}_{r} 
&&& \bar{d}_{r} 
\end{vmatrix}
\end{equation}
We add the $(r+i)^{th}$ column  to the $i^{th}$ column, then we subtract the $1/2$ of the  $i^{th}$ column from the $(r+i)^{th}$ column for all $i=1,\dots,r$. 
This results  the  determinant on the left hand side of \eqref{eq-det2}. 
From the first $r$ columns we get a $2^{r}$ factor. Then we switch the $(2i)^{th}$   and $(r+2i)^{th}$ columns for all $1 \leq  i \leq r/2$. This yields the determinant on the right hand side of \eqref{eq-det2}, which has zeroes in the even and odd rows of the first and last $r$ columns, respectively.
\begin{equation}\label{eq-det2}
\begin{vmatrix}
2 & & &
&  0 & &
\\
0 & 2 &   & 
& 
\bar{d}_{1} & 0 &  & 
\\
2\bar{d}_{2} & 0 &  \ddots &   
& 
0  & \bar{d}_{1}  & \ddots & 
\\
0  & 2\bar{d}_{2} &  \ddots  & 2 
& \bar{d}_{3} &  0 & \ddots &  0
\\
\vdots  & 0  & \ddots  & 0
& 
\vdots & \bar{d}_{3}  & \ddots & \bar{d}_{1}
\\
 & \vdots  & &  2\bar{d}_{2}
& 
 & \vdots & & 0 
\\
& & \ddots & \vdots 
& & & \ddots  & \vdots 
\end{vmatrix}
=
(-1)^{\lfloor \frac{r}{2} \rfloor}2^{r}
\begin{vmatrix}
1 & & 
&   & 0 &
\\
0 & 0 &   & 
& 
\bar{d}_{1} & 1 & 
\\
\bar{d}_{2} & \bar{d}_{1} & 1 &   
& 0  & 0 & 0 & 
\\
0 & 0  & 0  & \ddots
 & \bar{d}_{3} &  \bar{d}_{2}  &  \bar{d}_{1}   
& \ddots
\\
\vdots & \bar{d}_{3} & \bar{d}_{2} & \ddots 
 & \vdots & 0 & 0 & \ddots
\\
 & \vdots   &  0
& 
\ddots &  & \vdots  & \bar{d}_{3} & \ddots
\\
& & \vdots & \ddots
&&  & \vdots & \ddots 
\end{vmatrix}.
\end{equation}
Moving the odd rows up and the even rows down yields a $2$-by-$2$ block determinant with zeroes in the off-diagonal blocks and a $(-1)^{r(r-1)/2}$-sign, which cancels the existing $(-1)^{\lfloor r/2 \rfloor}$-sign.
We expand this determinant with respect to the first and last rows. 
These rows contain  only zeroes except the first row has $1$ in the first column  and the last row has $\bar{d}_{r}$ in the last column. After expansion the two diagonal block  become identical, hence
\begin{equation*}
\\
2^{r}\begin{vmatrix}
1 & & & & & 0  & 0 & 0 & \cdots & 0
\\
\bar{d}_{2} & \bar{d}_{1} & 1 & &  & 0 & & &  & 0
\\
\bar{d}_{4} & \bar{d}_{3} & \bar{d}_{2} & \ddots
& \vdots  & \vdots &&& & \vdots 
\\
\vdots & \vdots &  \vdots & \ddots & \bar{d}_{r-3} & 0 & & &  & 0 
\\
 &  &  &   & \bar{d}_{r-1}  & 0 & 0 &  0 & \cdots & 0
\\
0 & 0 & 0 & \cdots &  0 
&
\bar{d}_{1} & 1 &  
\\
0 &  &  &  & 0 & 
\bar{d}_{3} & \bar{d}_{2} & \bar{d}_{1} & 
\\
\vdots  & &   &  &   \vdots & \vdots & \bar{d}_{4} & \bar{d}_{3} & \ddots & \vdots
\\
0 &  &  &   &  0 & & \vdots & \vdots & \ddots & \bar{d}_{r-2} 
\\
0 & 0 &  0&  \cdots &  0 & & & &  &\bar{d}_{r}
\end{vmatrix}
=
2^{r} \bar{d}_{r}
\begin{vmatrix}
\bar{d}_{1} & 1 &  &   &   &\\
\bar{d}_{3} & \bar{d}_{2} & \bar{d}_{1} & 1 &  
\\
\bar{d}_{5} & \bar{d}_{4} & \bar{d}_{3} & \bar{d}_{2} & \bar{d}_{1} & \\
\vdots & \ddots & \ddots & \ddots & \ddots & 
 \\
&  & d_{r} & \bar{d}_{r-1} & \bar{d}_{r-2} & \bar{d}_{r-3} 
\\
&&&  & \bar{d}_{r} & \bar{d}_{r-1}
\end{vmatrix}^{2}.
\end{equation*}
Note that $C(\mathcal{E};s/2) = \bar{d}_{r} = d_{r}(\mathcal{E};s/2)$. Finally, by the relation 
$$
C(\mathcal{E}\otimes \mathcal{E};s) 
= 
C(\wedge^{2}\mathcal{E} \oplus S^{2}\mathcal{E};s) 
= 
C(\wedge^{2}\mathcal{E};s) C( S^{2}\mathcal{E};s) 
= 
2^{r}C(\mathcal{E};s/2) C(\wedge^{2}\mathcal{E};s)^{2}
$$ 
we are able to identify the $C(\wedge^{2}\mathcal{E};s)$ part in $C(\mathcal{E}\otimes \mathcal{E};s)$ to be \eqref{eq-wedge2}.
\end{proof}

\begin{remark}
We can also compute $C(\wedge^{r-2}\mathcal{E}; s)$ from $C(\wedge^{2}\mathcal{E}; s)$ by the duality
\begin{gather*}
C(\wedge^{r-2}\mathcal{E}; s) = \prod_{1\leq i_{1}<
\dots < i_{r-2} \leq r} (s + \alpha_{i_{1}}+\dots +\alpha_{i_{r-2}}) = \prod_{1\leq j_{1}< j_{2} \leq r} (s + c_{1}(\mathcal{E}) - \alpha_{j_{1}} - \alpha_{j_{2}})
\\
= \sum_{k=0}^{r(r-1)/2} (s + c_{1}(\mathcal{E}))^{r(r-1)/2-k} (-1)^{k} c_{k}(\wedge^{2}\mathcal{E})
=
(-1)^{r(r-1)/2}C(\wedge^{2}\mathcal{E}; -(s + c_{1}(\mathcal{E}))).
\end{gather*}
\end{remark}

\subsection{Resultant of monic polynomials}

Comparing the two formulas for the Chern class of the tensor product leads to another formula for the resultant of two polynomials, when one of them is monic.
\begin{thm}
Let $A(t) = \sum_{k=0}^{r}a_{k}t^{r-k}$ and $B(t) = \sum_{\ell=0}^{q} b_{\ell} t^{q-\ell}$ with $b_{0}=1$, i.e.~$B$ is monic. Then
$$
\res(A(t),B(t),t) = \det\left( \sum_{k=0}^{r}(-1)^{k}a_{k}\Lambda(b)^{r-k} \right),
$$
where $\Lambda(b)$ equals the matrix \eqref{eq-Lambda} with coefficients $b_{1},\dots ,b_{q}$ of $B$ in the first column.
\end{thm}
\begin{proof}
Assume that $\alpha_{1},\dots ,\alpha_{q}$ and $\beta_{1},\dots ,\beta_{q}$ are the roots of $A(t)$ and $B(t)$, respectively. Then $A(t) = a_{0}\prod_{i=1}^{r} (t - \alpha_{i})$, hence $e_{k}(\alpha) = (-1)^{k}\frac{a_{k}}{a_{0}}$ for $k=0,1,\dots ,r$. 
Similarly, $B(t) = \prod_{j=1}^{q} (t-\beta_{j})$, hence $e_{\ell}(\beta) = (-1)^{\ell} b_{\ell}$ for $\ell = 0,1,\dots ,q$. 
Substituting $s=0$, $u_{1} = \alpha_{1},\dots ,u_{r} = \alpha_{r}$, $v_{1}=-\beta_{1},\dots ,v_{q} = -\beta_{q}$ in Lemma \ref{lm-0} yields
\begin{gather*}
\res(A(t), B(t),t) = a_{0}^{q}\prod_{i=1}^{r}\prod_{j=1}^{q}(\alpha_{i}-\beta_{j})
=
a_{0}^{q} \det\left( \sum_{k=0}^{r} e_{k}(\alpha) \Lambda(e(-\beta))^{r-k} \right)
=
\\
=
a_{0}^{q} \det\left( \sum_{k=0}^{r} (-1)^{k}\frac{a_{k}}{a_{0}} \Lambda(b)^{r-k} \right)
= 
\det\left( \sum_{k=0}^{r}(-1)^{k}a_{k}\Lambda(b)^{r-k} \right).
\qedhere
\end{gather*}
\end{proof}

\appendix
\section{Performance of the implementations}
\label{s:tests}

By the results of \cite{iena2016different} (comparison table on page 3)   the multiplicativity of the Chern character implementation is the fastest among the four methods mentioned in Subsection \ref{ss:existing methods}, therefore we will test our methods against  it.
We use the test function $TST(N)$ provided in Appendix A of \cite{iena2016different}, which computes  the Chern classes of tensor products $\mathcal{E}\otimes \mathcal{F}$ for all pairs of vector bundles $(\mathcal{E},\mathcal{F})$  with $rk (\mathcal{E}) \cdot rk (\mathcal{F}) = N$. 
The results are shown in Figure \ref{fig:1}.

\begin{figure}[h!]
\centering
\begin{tikzpicture} 
\begin{semilogyaxis}[
	log ticks with fixed point
	,xlabel=$N$
	,ylabel=time (ms) 
	,scale only axis
	,legend pos=north west
	,xmin=1, xmax=40
	,ymin=1, ymax=1e+8
	,width=0.6*\textwidth 
	,yticklabels={1, $10^{1}$, $10^{2}$, $10^{3}$,$10^{4}$, $10^{5}$,$10^{6}$, $10^{7}$ , $10^{8}$}
	] 

\addplot [mark=*,red,thick] coordinates
{
(1,1)
(2,1)
(3,10)
(4,1)
(5,1)
(6,10)
(7,10)
(8,10)
(9,10)
(10,10)
(11,10)
(12,50)
(13,10)
(14,40)
(15,80)
(16,130)
(17,30)
(18,370)
(19,60)
(20,960)
(21,850)
(22,700)
(23,210)
(24,9850)
(25,4920)
(26,3550)
(27,14820)
(28,45160)
(29,1320)
(30,180630)
(31,2460)
(32,246120)
(33,170910)
(34,64030)
(35,935460)
(36,2450290)
(37,13470)
(38,248750)
(39,1589360)
(40,10178410)
};


\addplot[mark=triangle,green,thick] coordinates
{
(1,1)
(2,1)
(3,1)
(4,10)
(5,1)
(6,10)
(7,1)
(8,10)
(9,10)
(10,10)
(11,10)
(12,20)
(13,20)
(14,30)
(15,40)
(16,80)
(17,30)
(18,170)
(19,40)
(20,400)
(21,330)
(22,330)
(23,90)
(24,2410)
(25,960)
(26,1010)
(27,2510)
(28,6270)
(29,310)
(30,20240)
(31,440)
(32,22910)
(33,14710)
(34,8050)
(35,68810)
(36,167720)
(37,1160)
(38,22440)
(39,73930)
(40,689200)
};


\addplot[mark=o,blue,thick] coordinates {
(1,1)
(2,1)
(3,1)
(4,1)
(5,1)
(6,10)
(7,1)
(8,1)
(9,1)
(10,10)
(11,1)
(12,10)
(13,1)
(14,1)
(15,20)
(16,20)
(17,1)
(18,50)
(19,1)
(20,130)
(21,100)
(22,30)
(23,10)
(24,910)
(25,530)
(26,80)
(27,750)
(28,3140)
(29,20)
(30,11470)
(31,30)
(32,13140)
(33,4660)
(34,390)
(35,67310)
(36,121990)
(37,40)
(38,920)
(39,22320)
(40,810640)
};

\legend{
	Multiplicativity of the 
	Chern character,
	Formula \eqref{eq-C},
	 Formula \eqref{eq-tensor-Chern}
	 } 
\end{semilogyaxis} 
\end{tikzpicture} 
\caption{Time it takes to compute $TST(N)$ in \textsc{Singular 4.1.2} on MacOS with Intel i7-4770HQ CPU (running at 3.2GHz) and 16GB RAM.}
\label{fig:1}
\end{figure}
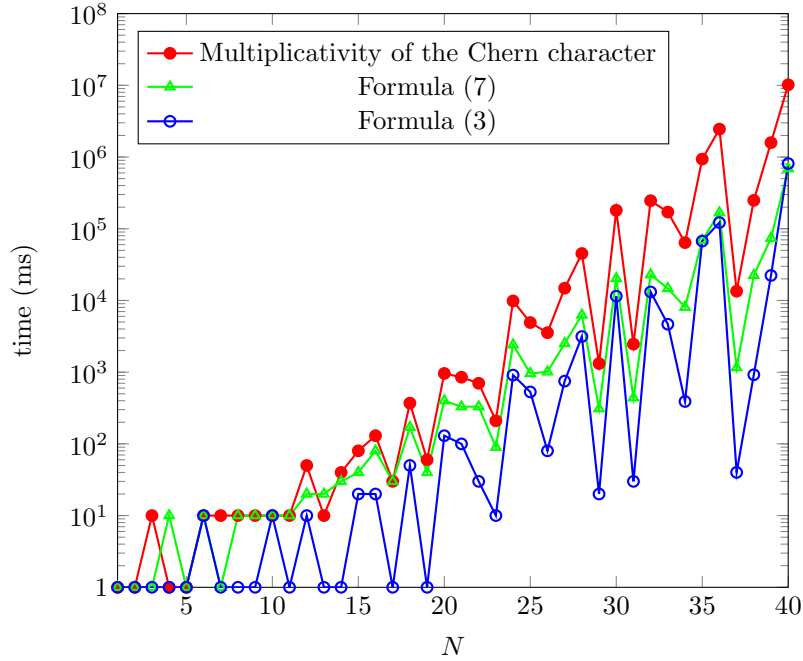

\section{Codes in Singular}

We provide the codes of our implementations. We mention that there exists a procedure in \textsc{Singular} for the resultant of two polynomials, called \texttt{resultant}, however we found it too slow, therefore we reimplemented it in our procedure \texttt{resultant2}.

\begin{verbatim}
LIB "general.lib"; // for the function binomial used in chProdRes; 
LIB "chern.lib"; // for the function chProd used in TST;

// Implementation of the first approach:

proc chProd1( int m, list x, int n, list y ) { 
    // m is rank of the bundle with x as list of Chern classes; 
    // n is rank of the bundle with y as list of Chern classes;
    int i;
    matrix  I_plus_Lambda[n][n]; 
    for(i=1; i<n; i++){
        I_plus_Lambda[i,i+1] = -1;     
    }
    for(i=1; i<=n; i++){
        I_plus_Lambda[i,1] = y[i] * t^i; // t is the variable of the Chern polynomial;
        I_plus_Lambda[i,i] = I_plus_Lambda[i,i] + 1; 
    }
    matrix S[n][n]; 
    for(i=1; i<=n; i++) {
        S[i,i]=1;  // setting S to identity matrix;
    } 
    for(i=1; i<=m; i++) {
        S =  x[i] * t^i + S * I_plus_Lambda; 
    }
    poly c_Tensor = det(S); // Chern polynomial of the tensor product;
    matrix c_Tensor_coeffs_mx = coeffs(c_Tensor,t); //coeffs of the Chern poly as matrix;
    list c_Tensor_coeffs_list; 
    for(i=1; i<=m*n; i++) {
        c_Tensor_coeffs_list[i] = c_Tensor_coeffs_mx[i+1,1]; //convert coeffs to list;
    } 
    return(c_Tensor_coeffs_list); // return the list of coefficients;
}

// Implementation of the resultant approach:

proc chProdRes( int m, list x, int n,  list y ) {
    // Resutant of polynomial A and B;
    proc resultant2 (poly A , poly B, t) {
        matrix coeff_A = coeffs(A,t); // coefficients of polynomial A;
        int p = size(coeff_A)-1; //  degree of A;
        matrix coeff_B = coeffs(B,t); // coefficients of polynomial B;
        int q = size(coeff_B)-1; // degree of B;
        matrix M[p+q][p+q];
        int i,j,k;
        k=0;
        for(i=1; i<=p; i++) {
            for(j=1; j<=q+1; j++) {
                M[k+j,i] = coeff_B[j,1];  
            }  
            k++;
        }
        k=0;
        for(i=p+1; i<=p+q; i++) {
            for(j=1; j<=p+1; j++) {
                M[k+j,i] = coeff_A[j,1];  
            }  
            k++;
        }
        return(det(M));
    }

    int i, j, k;
    poly d=0;
    poly C = t^m; // C polynomial of the first bundle;
    for(i=1; i<=m; i++) {
        C = C + x[i] * t^(m-i);
    } 
    poly D = t^n; // D polynomial of the second bundle;
    for(k=1; k<=n; k++) {
        d = binomial(n,k) * s^k;
        for(j=1; j<=k; j++) {
            d = d + binomial(n-j,k-j) * s^(k-j) * y[j];
        }
        D = D + (-1)^k * d * t^(n-k);
    }
    poly C_Tensor = resultant2(D,C,t); // C polynomial of the tensor product;
    // Computing the homogeneous parts:
    matrix C_Tensor_mx = coeffs(C_Tensor,s); // coefficients of C_Tensor;
    list C_Tensor_list;
    for(i=0; i<m*n; i++) {
    	j = m*n-i;
        C_Tensor_list[i+1] = C_Tensor_mx[j,1]; // convert coefficients to list;
    }
    return(C_Tensor_list);
}

proc PRD(int N) { 
    // generates list of all pairs (m,n) with m * n = N;
    list L, l;
    int i, j;
    for(i=1; i<=N; i++) {
        j = N div i;
        if(i*j == N) {
            l = list(j,i);
            L = insert(L,l);
        } 
    }
    return(L);
}

proc TST(int N) {
    // the test procedure that computes the Chern classes of 
    // all tensor products of rank N;
    ring r = 0,(x(1..N), y(1..N),s,t), dp;
    int j, m, n, sz;
    list L, x, y;
    string str;
    L = PRD(N); // list of pairs (m, n) with m * n=N;
    sz = size(L);
    for(j = 1; j <= sz; j++) {
        m = L[j][1];
        n = L[j][2];
        // for the multiplicativity of the Chern character:
        //chProd( m, list( x(1..m) ),  n, list( y(1..n) ) );
        // for the first approach:
        chProd1(m, list( x(1..m) ), n, list( y(1..n) ) );
        // for the resultant approach:
        //chProdRes(m, list( x(1..m) ), n, list( y(1..n) ) );
    } 
}

// Test begin
list ListOfTimes;
for(int k=1; k<=40;k++){
    timer=0;
    system("--ticks-per-sec",1000); 
    TST(k);
    ListOfTimes[k] =  timer;
    print(timer);  
}
// Test end

// Example:
int N = 10;
ring r = 0, ( x(1..N), y(1..N), t, s ), dp;
TST(N);

// Example:
list X = x(1..3); // list of Chern classes of a rank 3 bundle;
list Y = y(1..5); // list of Chern classes of a rank 5 bundle;
chProd1(3, X, 5, Y); 
chProdRes(3, X, 5, Y); 
\end{verbatim}


\bibliographystyle{abbrv}
\bibliography{Bibdesk}

\end{document}